\documentclass[final]{siamart1116}
\usepackage{multirow}
\usepackage{times} 
\usepackage{amsmath,lipsum} 
\usepackage{amssymb,amsfonts}  
\usepackage{latexsym,theorem,epsfig}
\usepackage{graphicx} 
\usepackage{dsfont}
\usepackage{mathtools}
\usepackage{subfigure}
\usepackage{lscape}
\usepackage{color}
\usepackage{float}
\usepackage{enumerate}
\usepackage[usenames,dvipsnames]{pstricks}
\usepackage{epsfig}
\usepackage{pst-grad} 
\usepackage{pst-plot} 
\usepackage{mathrsfs}
\usepackage{extarrows}
\usepackage[percent]{overpic}
\usepackage{epstopdf}
\usepackage[misc]{ifsym}
\usepackage{caption}
\usepackage{lipsum}
\usepackage{amsfonts}
\usepackage{graphicx}
\usepackage{epstopdf}
\usepackage{algorithmic}

\usepackage[english]{babel}
\usepackage{threeparttable, tablefootnote}

\usepackage[a4paper, total={6in, 8in}]{geometry}

\usepackage{array}
\newcolumntype{L}[1]{>{\raggedright\let\newline\\\arraybackslash\hspace{0pt}}m{#1}}
\newcolumntype{C}[1]{>{\centering\let\newline\\\arraybackslash\hspace{0pt}}m{#1}}
\newcolumntype{R}[1]{>{\raggedleft\let\newline\\\arraybackslash\hspace{0pt}}m{#1}}

\newsiamthm{remark}{Remark}
\newsiamthm{example}{Example}

\allowdisplaybreaks

\def\xb{\mathbf{x}}

\def\sl#1{{\color{red}#1}}

\DeclareMathOperator*{\argmin}{arg\,min}
\DeclareMathOperator*{\argmax}{arg\,max}
\DeclareMathOperator{\spn}{span}

\numberwithin{theorem}{section}

\newcommand{\TheTitle}{Optimal Algorithms for Distributed Optimization} 
\newcommand{\TheAuthors}{C.A.\ Uribe, S. Lee, A. Gasnikov and A. Nedi\'{c}}

\headers{\TheTitle}{\TheAuthors}

\title{{\TheTitle}\thanks{\textbf{An extended version of this work can be found in~\cite{Uribe2018new}}. \newline The work of A. Nedi\'{c} and C.A.\ Uribe was supported by the National Science Foundation under grant no. CPS~15-44953. The work of A. Gasnikov was supported by RFBR 18-29-03071 mk and MD-1320.2018.1.} }

\author{
	C.A. Uribe\thanks{ECE Department and Coordinated Science Laboratory, University of Illinois (\email{cauribe2@illinois.edu}).}
	\and
	S. Lee\thanks{Yahoo! Research  (\email{soominl@yahoo-inc.com}).}
	\and
	A. Gasnikov\thanks{Moscow Institute of Physics and Technology, Dolgoprudnyi 141700, Moscow
		Oblast, Russia.
		(\email{gasnikov@yandex.ru}).} 
	\and
	A. Nedi\'{c} \footnotemark[4] \thanks{ECEE Department, Arizona State University (\email{angelia.nedich@asu.edu}).}
}

\begin{document}
	\maketitle
	\begin{abstract}
		In this paper, we study the optimal convergence rate for distributed convex optimization problems in networks. We model the communication restrictions imposed by the network as a set of affine constraints and provide optimal complexity bounds for four different setups, namely: the function $F(\xb) \triangleq \sum_{i=1}^{m}f_i(\xb)$ is strongly convex and smooth, either strongly convex or smooth or just convex. Our results show that Nesterov's accelerated gradient descent on the dual problem can be executed in a distributed manner and obtains the same optimal rates as in the centralized version of the problem (up to constant or logarithmic factors) with an additional cost related to the spectral gap of the interaction matrix. Finally, we discuss some extensions to the proposed setup such as proximal friendly functions, time-varying graphs, improvement of the condition numbers.
	\end{abstract}
	
	\section{Introduction}
	
	The study of distributed algorithms can be traced back to classic papers from the 70s and 80s~\cite{aum76,bor82,tsi84,gen86,coo90,deg74,gil93}. The adoption of distributed optimization algorithms on several fronts of applied and theoretical machine learning, robotics, and resource allocation has increased the attention on such methods in recent years \cite{xia06,rab04,kon15,kra13,ned17e}. The particular flexibilities induced by the distributed setup make them suitable for large-scale problems involving large quantities of data  \cite{bot10,boy11,aba16,ned16w,ned15}. 
	
	Initial algorithms for distributed optimization such as distributed subgradient methods were shown successful for solving optimization problems in a distributed manner over networks \cite{ned09,ned09b,ram10,ned13}. Nevertheless, these algorithms are particularly slow compared with their centralized counterparts. Moreover, from the optimization perspective, proposing and analyzing algorithms with equivalent performance to their centralized counterparts have always been a priority. In a recent stream of literature, new distributed methods that achieve linear rates for strongly convex and smooth problems have been proposed in \cite{ned17,ned16w,shi15,lan17,sca17}. 
	
	One can identify three main approaches to the study of distributed algorithms. In \cite{ned16w,ned17,ned17r}, the focus is achieving linear convergence rate for strongly convex and smooth problems. These results require some minimal information about the topology of the network and provide explicit statements about the dependency of the convergence rate on the problem parameters. Specifically, polynomial scalability is shown with the network parameter for particular choices of small enough step-sizes and even uncoordinated step-sizes are allowed \cite{ned17r}. One particular advantage of this approach is that can handle time-varying and directed graphs. Nevertheless, optimal dependencies on the problem parameters and tight convergence rate bounds are far less understood for such algorithms. More recently, in \cite{ned17} new methods were proposed where it was shown that $O((m^2 + \sqrt{L/\mu}m)\log \varepsilon^{-1})$ iterations are required to find an $\varepsilon$ solution to the optimization problem when the function is $\mu$-strongly convex and $L$-smooth, where $m$ is the number of nodes in the fixed undirected network. Recently in \cite{sun17}, a unified approach has been proposed for analysis of the convergence rate of distributed optimization algorithms via a semidefinite programming characterization. There, too, the focus on strongly convex and smooth functions. This approach provides an innovative procedure to numerically certify worst-case rates of a plethora of distributed algorithms, which can be useful to fine-tune parameters in existing algorithms based on feasibility conditions of a semidefinite program. Nevertheless, the questions on the specific dependency on the problem parameters and optimality certifications of the algorithms remain open. Similarly, in \cite{jak17}, another unifying approach has been proposed, that encompasses several existing algorithms such as those in \cite{shi15,qu17}. This newly proposed general method is able to recover existing rates and achieves an $\varepsilon$ precision in $O(\sqrt{L/ (\mu \lambda_2)} \log \varepsilon^{-1})$ iterations, where $\lambda_2$ is the second largest eigenvalue of the interaction matrix. Finally, a third approach was recently introduced in \cite{sca17}, where the first optimal algorithm for distributed optimization problems was proposed. This new method achieves an $\varepsilon$ precision in $O(\sqrt{L/\mu}(1 + \tau / \sqrt{\gamma}) \log \varepsilon^{-1})$ iterations for \mbox{$\mu$-strongly} convex and $L$-smooth problems, where $\tau$ is the diameter of the network and $\gamma$ is the normalized eigengap of the interaction matrix. Even though extra information about the topology of the network is required, the work in \cite{sca17} provides a coherent understanding of the optimal convergence rates and its dependencies on the communication network.
	
	In this paper, we follow the approach in \cite{sca17} to study the problem of distributed optimization over networks. Particularly, we consider the following optimization problem
	\begin{align}\label{main_problem}
	\min_{\xb \in \mathbb{R}^n} F(\xb) & \qquad \text{where } \qquad F(\xb)\triangleq \sum\limits_{i=1}^{m}f_i(\xb),
	\end{align}
	where each $f_i: \mathbb{R}^n \to \mathbb{R}$ is a convex function. Moreover, this problem needs to be solved in a distributed manner where each function $f_i$ is known by an agent $i$ only and agents can interact with other agents over a network. We provide algorithms with provable optimal convergence rates for the cases where the function $F(\xb)$ is strongly convex and smooth, only strongly convex and Lipschitz, only smooth or just convex and Lipschitz. In each of the cases, we will study optimal algorithms for solving convex optimization problems with affine constraints \cite{ani17,che16,gas17}. Initially, we will consider the case when the construction of an explicit dual problem is possible. Later, we will remove this assumption. Moreover, we consider the case of proximal friendly functions. Additional extensions to time-varying graphs and improving the condition number are discussed.
	
	Our results match known optimal complexity bounds for centralized convex optimization (obtained by classical methods such as Nesterov's Fast Gradient Method \cite{nes83}), with an additional cost induced by the network of communication constraints. This extra cost appears in the form of a multiplicative term proportional to the square root of the spectral gap of the interaction matrix. In summary, our main results provide an algorithm that achieves $\varepsilon$ accuracy on any fixed, connected and undirected graph according  \cref{tab:summary_other}, where universal constants are hidden for simplicity.

\bgroup
\def\arraystretch{1.7}
	\begin{table}[ht]
	\centering
	\resizebox{\columnwidth}{!}{
	\begin{tabular}{|c|C{1.5cm}|c|c|c|c|} \hline
		 \bf Approach &\multicolumn{1}{|C{1.5cm}|}{\bf Reference} & \multicolumn{1}{|C{2.7cm}|}{\bf $\mu$-strongly convex and $L$-smooth}    & \multicolumn{1}{|C{2.2cm}|}{\bf  $\mu$-strongly convex and \mbox{$M$-Lipschitz}}   & \multicolumn{1}{|C{1.5cm}|}{\bf $L$-smooth}   & \multicolumn{1}{|C{2.1cm}|}{\bf \mbox{$M$-Lipschitz}}   \\ \hline
		 Centralized	&\cite{Nemirovskii1983}	& $\tilde{O}\left( \sqrt{\frac{L}{\mu}} \right)$ &$\tilde{O}\left( \frac{M^2}{\mu \varepsilon} \right) $ &   $\tilde{O}\left( \sqrt{\frac{L}{\varepsilon} }\right) $  &  $\tilde{O}\left( \frac{M^2}{\varepsilon^2}\right) $ \\\hline
		\multirow{7}{*}{\parbox[c]{1.5cm}{\centering Primal}} 
		&\cite{Qu2017} &$\tilde O\left( \left( \frac{L}{\mu} \right) ^{5/7}m^3\right) $ & $-$ & $\tilde O\left( \frac{1}{\varepsilon^{5 / 7}}\right)^{\text{***}}$ & $-$\\
		&\cite{ned17} & $\tilde{O}\left( m^2 + \sqrt{\frac{L}{\mu}}m\right)$ & $-$& $o\left( \frac{1}{\varepsilon}\right) $& $o\left( \frac{1}{\varepsilon}\right)^\text{**} $\\
		&\cite{ols14} & $-$ & $-$& $-$& $\tilde{O}\left( \frac{M^2m}{\varepsilon^2}\right) $\\
		&\cite{Doan2017} & $\tilde{O}\left( \frac{L}{\mu}m^2\right)$ &$-$ & $-$ & $-$ \\
		&\cite{Lakshmanan2008} 		&  $-$&$-$ & $\tilde{O}\left( \frac{L }{\varepsilon} m^3\right)$ & $-$  \\ 
		&\cite{Necoara2013}	& $\tilde{O}\left( \frac{L}{\mu}m^4\right)$ &$-$ &   $\tilde{O}\left( \frac{L }{\varepsilon} m^4\right)$  &$-$\\
		& \cite{jak17} & $\tilde{O}\left( \sqrt{\frac{L}{\mu}}m^2\right)^{\text{****}}$ & & & \\ \hline
		\multirow{2}{*}{\parbox[c]{1.5cm}{\centering Dual Friendly}} &\cite{sca17}$^\text{*}$	&  $\tilde{O}\left( \sqrt{\frac{L}{\mu}}m \right)$ & $-$ & $-$ & $-$\\
		&This paper$^\text{*}$& $\tilde{O}\left( \sqrt{\frac{L}{\mu}}m \right)$ &$\tilde{O}\left( \sqrt{\frac{M^2}{\mu \varepsilon}} m\right) $ &   $\tilde{O}\left( \sqrt{\frac{L }{\varepsilon} }m\right) $  &  $\tilde{O}\left( \frac{M}{\varepsilon}m\right) $\\ \hline
		\multicolumn{6}{L{\columnwidth}}{$^\text{*}$ The function $F(x)$ is assumed dual friendly. \hspace{7cm}  $^\text{**}$ Additionally, it is assumed functions are proximal friendly. \hspace{7cm}$^\text{***}$ An iteration complexity of $\tilde{O}(\sqrt{1 / \varepsilon} ) $ is shown if the objective is the composition of a linear map and a strongly convex and smooth function. \hspace{6cm} $^\text{****}$ A linear dependence on $m$ is achieved if $L$ is sufficiently close to $\mu$.}
	\end{tabular}
	}
	\caption{\textbf{Iteration Complexity of Distributed Optimization Algorithms}. The iteration complexity refers to the required number of oracle calls. For distributed algorithms based on primal iterations this translates to computations of gradients of the local functions for each of the agents. On the other side, for dual based algorithms, the complexity refers to computations of the gradient of the Lagrange dual function, which in translates to the number of communication rounds in the network. }
	\label{tab:summary_other}
	\end{table}
	
	This paper is organized as follows: Section \ref{sec:problem} introduces the problem of distributed optimization over networks. Section \ref{sec:prelim} presents a series of definitions and auxiliary results that will help in the exposition of the main results. Section \ref{sec:main} provides our main result regarding the optimal convergence rates for the solution of distributed optimization problems for the different variations of smoothness and strong convexity properties. In Section \ref{sec:consensus} we study the specific cost of having a distributed setup where we use as a guiding example the consensus problem, for which we find an optimal algorithm regarding the condition numbers and the properties of the graph. In Sections \ref{sec:no_dual} and \ref{sec:prox}, \ref{sec:condition} we discuss extensions when the studied function is not \textit{dual-friendly} but proximal operations are easy to compute. Addutionally we discuss how to improve the condition numbers. In Section \ref{sec:discussion} we provide some remarks of less developed extensions for the distributed optimization setup, particularly, we discuss time-varying/directed graphs, the communication time versus the computation time and working with general $p$-norms, with $p\geq1$. Finally, in we present some conclusions.
	
	\vspace{1em}
	\noindent\textbf{Notation: }
	Generally, we will use the superscript $i$ or $j$ to denote agents, while the subscript $k$ will denote time (or iterations of an algorithm). We denote as $[A]_{ij}$ the entry of the matrix $A$ at its $i$-th row and $j$-th column and $I_{n}$ is the identity matrix of size $n$. We denote the largest singular value of a matrix $A$ as $\sigma_{\max}(A) = \lambda_{\max}(A^TA) = \max \{\lambda \mid \exists x \neq 0, A^TAx = \lambda x \}$. We use $\sigma_{\min}(A) = \min \{\lambda >0 \mid \exists x \neq 0, A^TAx = \lambda x \}$ {as the smallest nonzero singular value}. Also, we define $\chi(A^TA) = \frac{\sigma_{\max}{(A)}}{\sigma_{\min}{(A)}}$. The notation $\tilde{O}$ is used to ignore logarithmic factors in an estimate.

	\section{Problem Statement}\label{sec:problem}
	
	In this section we will introduce the problem of distributed optimization. Initially, let us introduce a stacked column vector ${{{x} = [x_1^T,x_2^T,\hdots,x_m^T]^T \in \mathbb{R}^{mn}}}$ to rewrite the problem in Eq. \eqref{main_problem} in an equivalent form as
	\begin{align*}
	\min_{x_1 = \hdots =x_m} {F(x)} & = \sum\limits_{i=1}^{m}f_i(x_i).
	\end{align*}
	
	Now, suppose that we are required to solve this problem in a distributed manner over a network. We model such a network as a fixed connected undirected graph $\mathcal{G} = (V,E)$, where $V$ is a set of {$m$} nodes and $E$ is a set of edges. The network structure imposes information constraints, namely, each node $i$ has access to the function $f_i$ only, {and a node can exchange information only with its immediate neighbors, i.e., a node $i$ can communicate with node $j$ if and only if $(i,j)\in E$}.
	
	We can represent the communication constraints imposed by the network by introducing a set of constraints to the original optimization problem in Eq.\cref{main_problem}. Initially, define the Laplacian matrix $\bar W{\in \mathbb{R}^{m\times m}}$ of the graph $\mathcal{G}$ as
	\begin{align*}
	[\bar W]_{ij} = \begin{cases}
	-1,  & \text{if } (i,j) \in E,\\
	\text{deg}(i), &\text{if } i= j, \\
	0,  & \text{otherwise }.
	\end{cases}
	\end{align*}
	where $\text{deg}(i)$ is the degree of the node $i$, i.e., the number of neighbors of the node. Finally, define the communication matrix, sometimes called interaction matrix as $W = \bar W \otimes I_n$, where $\otimes$ indicates the Kronecker product.
	
	One can verify that $W$ is a nonnegative semidefinite matrix, with the following properties:
	\begin{itemize}
		\item $W{x} = 0$ if and only if {$x_1 = \hdots = x_m$}.
		\item $\sqrt{W}{x} = 0$ if and only if {$x_1 = \hdots = x_m$}.
		\item $\sigma_{\max} (\sqrt{W}) = \lambda_{\max} (W)$.
	\end{itemize}
	Therefore, one can equivalently rewrite the problem in Eq. \eqref{main_problem} as:
	\begin{align}\label{consensus_problem2}
	\min_{\sqrt{W} x=0} {F(x)} & = \sum\limits_{i=1}^{m}f_i(x_i),
	\end{align}
	where the constraint $\sqrt{W} x=0$ is equivalent to \mbox{$x_1 = \hdots  = x_m$} given that $\ker (\sqrt{W}) = \spn(\boldsymbol{1})$.

	\section{Preliminaries}\label{sec:prelim}
	
	Before presenting our main results, in this section, we will provide a set of definitions and preliminary information that we will use throughout the paper in the remaining sections. 
	
		\begin{remark}
		For simplicity of exposition, we will present our results considering the $2$-norm only, which is called the Euclidean setup. Nevertheless, in Section \ref{sec:discussion} we will point out the generalization to other norms.
	\end{remark}
	
	\begin{definition}\label{def:strong}
		A function $f(\cdot)$ is a $\mu$-strongly convex function if for any $x,y$ it holds that
		\begin{align*}
		f(y) \geq f(x) + \left\langle \nabla f(x),y-x \right\rangle + \frac{\mu}{2}\|x-y\|^2_2,
		\end{align*}
		where $\nabla f(x)$ is any subgradient of $f(\cdot)$ at $x$. 
	\end{definition}
	
	Particularly, if each of the functions $f_i(x_i)$ in the problem in Eq.\eqref{consensus_problem2} is $\mu_i$-strongly convex then $F(x)$ is $\mu$-strongly convex, with $\mu = \min\limits_{i=1,\hdots,m} \mu_i$.
	
	\begin{definition}\label{def:smooth}
		A function $f( \cdot )$ has $L$-Lipschitz continuous gradient if it is differentiable and its gradient satisfies the Lipschitz condition, i.e. for any $x,y$
		\begin{align*}
		\|\nabla f(x) - \nabla f(y)\|_2 & \leq L \|x-y\|_2
		\end{align*}
		A function having $L$-Lipschitz continuous gradient is also referred as $L$-smooth. 
	\end{definition}
	
	If each of the functions $f_i(x_i)$ in the problem in Eq.\eqref{consensus_problem2} has $L_i$-Lipschitz continuous gradients then $F(x)$ has $L$-Lipschitz continuous gradient with $L = \max\limits_{i=1,\hdots,m} L_i$. 
	
	Our main algorithmic tool will be Nesterov's Fast Gradient Method (FGM) \cite{nes13}. Next, in \cref{alg:nesterov} we state one variant of the FGM method for a $\mu$-strongly convex and $L$-smooth function $F(x)$. Other variants of this method can also be found in \cite{nes13,bec09,lan11}.
	
	\begin{algorithm}
		\caption{Nesterov's Fast Gradient Method}
		\label{alg:nesterov}
		\begin{algorithmic}[1]
			\STATE{Choose $y_0 = x_0 \in \mathbb{R}^n$}
			\WHILE{ stopping criteria }
			\STATE{$x_{k+1} = y_k - \frac{1}{L} \nabla f (y_k)$}
			\STATE{$y_{k+1} = x_{k+1} + \frac{\sqrt{L \vphantom{\mu}} - \sqrt{\vphantom{L}\mu}}{\sqrt{L}+ \sqrt{\vphantom{L}\mu}}(x_{k+1} - x_k)$}
			\ENDWHILE
			\RETURN $x_k$
		\end{algorithmic}
	\end{algorithm}
	
	Specifically it holds for \cref{alg:nesterov} that 
	\begin{align*}
	f(x_k) - f^* \leq L\left( 1- \sqrt{\frac{\mu}{L}} \right)^k\|x_0 - x^*\|_2^2, 
	\end{align*}
	where $f^*$ denotes the minimum value of the function $f(x)$ and $x^*$ is its minimizer. This result implies a geometric convergence rate, where one can find a solution that is $\varepsilon$ close to the optimal in 
	\begin{align}\label{rate_nest}
	N = O\left(\sqrt{\frac{L}{\mu}} \log \left( \frac{L\|x_0 -x^*\|^2_2}{\varepsilon} \right) \right)
	\end{align}
	iterations.
	
	In what follows we will consider a $\mu$-strongly convex and $L$-smooth function $f(x)$ and the general optimization problem with linear constraints
	\begin{align}\label{main2}
	\min_{Ax=0} f(x)
	\end{align} 
	where the Lagrangian dual problem of Eq. \eqref{main2} is
	\begin{align*}
	\max_y \left\lbrace \min_{{x}} \left( f(x) - \left\langle A^Ty,x \right\rangle \right) \right\rbrace.
	\end{align*}
	
	Now, rewrite the Lagrangian dual problem in its equivalent form as a minimization problem
	\begin{align}\label{dual_problem22}
	\min_y \varphi(y) & \qquad \text{where} \qquad \varphi(y) = \max_x \left\lbrace \left\langle A^Ty,x \right\rangle -f(x) \right\rbrace,
	\end{align}
	and $\varphi(y)$ is $\mu_\varphi = \frac{\sigma_{\min}(A)}{L}$-strongly convex in $\ker(A^T)^{\perp}$ and has $L_\varphi = \frac{\sigma_{\max}(A)}{\mu}$-Lipschitz continuous gradients.
	
	From the Demyanov-Danskin's theorem (see Proposition $4.5.1$ in \cite{Bertsekas2003})  it follows that \mbox{$\nabla \varphi(y) = Ax^*(A^Ty) $} where $x^*(A^Ty)$ is the unique solution to the inner maximization problem  
	\begin{align*}
	x^*(A^Ty) & = \argmax_x \left\lbrace \left\langle A^Ty,x \right\rangle -f(x) \right\rbrace.
	\end{align*}
	
%
	In order to find $x^*(A^Ty)$ one can use optimal (randomized) numerical methods \cite{nes13,nes94}. Nonetheless, initially we will assume that we have access to $x^*(A^Ty)$ explicitly. Later in Section~\ref{sec:no_dual} we will provide convergence rate estimates this assumption does not hold. 
	
	\begin{definition} 
		A function $f(x)$ is \textit{dual-friendly} if when considering the optimization problem in Eq.\eqref{main2} one has immediate access to an explicit solution $x^*(A^Ty)$ to the dual subproblem or it can be computed efficiently.
	\end{definition}
	
		In general, the matrix $A$ might not be full row rank. Then, the dual problem in Eq.\eqref{dual_problem22} can have multiple solutions of the form $y^* + \ker(A^T)$. If the solution is not unique, we will choose $y^*$ with the smallest norm solution. Moreover, we will denote its norm as $R = \|y^*\|_2$ and assume that $R < \infty$. 
		
		Note that we will use the result in Eq.\cref{rate_nest} applied to the dual problem in Eq.\eqref{main2}, which is not strongly convex in the ordinary sense (in the whole space). Nevertheless, since we will select $y_0 = x_0$ in \cref{alg:nesterov} as the initial condition, we will work in the linear space of gradients on which we have strong convexity.
	
	We will be interested in finding solutions to the problem in Eq.\eqref{main2} that are arbitrarily close to an optimal solution, both in terms of the error of the primal problem and the linear constraints feasibility. For this, we introduce the following definition.
	\begin{definition}
		We say that a point $\hat x$ is an $(\varepsilon,\tilde{\varepsilon})$-optimal solution for the problem in Eq.\eqref{main2} if the following condition holds
		\begin{align*}
		f(\hat x ) - f(x^*) \le  \varepsilon \qquad \text{and} \qquad \|A\hat x\|_2 \leq \tilde{\varepsilon}.
		\end{align*} 
		where $f(x^*)$ denotes the optimal function value the problem in Eq.\eqref{main2}. Moreover note that $\left| f(x^*(A^Ty) ) - f(x^*) \right| \le \|y\|_2\|Ax^*(A^Ty)\|_2 $
	\end{definition}

	
	To solve the dual problem in Eq.\eqref{dual_problem22} one can use the FGM algorithm that  specializes to \cref{alg:nesterov2}.
	\begin{algorithm}[H]
		\caption{Nesterov's Fast Gradient Method on the Dual Problem}
		\label{alg:nesterov2}
		\begin{algorithmic}[1]
			\STATE{Choose $y_0 = \tilde{y}_0 =0 \in \mathbb{R}^n$}
			\WHILE{ $\|y_k\|_2 \|Ax^*(A^Ty_k)\|_2 \geq \varepsilon \text{ and } \|Ax^*(A^Ty_k)\|_2 \geq \tilde{\varepsilon}$ }
			\STATE{$y_{k+1} = \tilde{y}_k - \frac{1}{L_\varphi} Ax^*(A^T\tilde{y}_k)$}
			\STATE{$\tilde{y}_{k+1} = y_{k+1} + \frac{\sqrt{L_\varphi} - \sqrt{\mu_\varphi\vphantom{L_\varphi}} }{\sqrt{L_\varphi}+ \sqrt{\mu_\varphi\vphantom{L_\varphi}}}(y_{k+1} - y_k)$}
			\ENDWHILE
			\RETURN $y_k$
		\end{algorithmic}
	\end{algorithm}

	From \cite{dev12}, we can immediately conclude that the number of oracle calls (calculations of $Ax$ and $A^Ty$) required for the \cref{alg:nesterov2} is
	\begin{align}\label{bound_dual}
	N & = O \left( \sqrt{\frac{L_\varphi}{\mu_\varphi}} \log \left(\max\left\lbrace 4 L_\varphi\frac{R^2}{\varepsilon} ,2 L_\varphi \frac{R}{\tilde{\varepsilon}}\right\rbrace  \right) \right) = \tilde{O}\left( \sqrt{\frac{L}{\mu} \chi(A^TA) }\right).
	\end{align}

	We will apply the FGM as described in \cref{alg:nesterov2} to the dual of different variations of the problem in Eq. \eqref{main_problem}. The next section presents the main results regarding the optimal convergence rates for the solution of the distributed optimization problem.
	
	\section{Main results}\label{sec:main}
	
	Our main results provide convergence rate estimates for the solution of the problem in Eq. \eqref{main_problem} (or Eq. \eqref{consensus_problem2}) for four different cases: 
	\begin{enumerate}
		\item $F(x)$ is $\mu$-strongly convex and $L$-smooth.
		\item $F(x)$ is $\mu$-strongly convex and $M$-Lipschitz.
		\item $F(x)$ is $L$-smooth.
		\item $F(x)$ is $M$-Lipschitz.
	\end{enumerate}

	\subsection{$F(x)$ is \mbox{$L$-smooth} and $\mu$-strongly convex}\label{sec:strong_smooth}
	
	This case is a specific version of the general problem presented in Eq.~\eqref{main2}. Particularly, when the linear constraints $Ax=0$ are $\sqrt{W}x=0$ and the function $f(x)$ corresponds to the function $F(x)$ as defined in Eq.\eqref{consensus_problem2}.
	
	Line $3$ and $4$ of \cref{alg:nesterov2} captures the interaction between agents and exchange of information in the distributed setting.  If we change of variables $\sqrt{W}y_k = z_k$ and \mbox{$\sqrt{W}\tilde y_k = \tilde z_k$}, then the resulting iterations can be executed in a distributed manner. The main observation is the interaction between agents is dictated by the term $Wx^*(\tilde z_k)$ which depends only on local information. Particularly, each node in the graph has its local variables $z_k^i$ and $\tilde z_k^i$, and to compute the value of these variables on the next iteration it only requires the information sent by the neighbors defined by the communication graph $\mathcal{G}$. Each node $i$ updates its local variables as
	\begin{subequations}\label{per_agent}
	\begin{align}
	z_{k+1}^i &= \tilde z_k^i - \frac{1}{L_\varphi} \sum_{j=1}^{m} W_{ij} x^{*}_j(\tilde z_k^j) \\
	\tilde z_{k+1}^i &= z_{k+1}^i + \frac{\sqrt{L_\varphi} - \sqrt{\mu_\varphi\vphantom{L_\varphi}}}{\sqrt{L_\varphi}+ \sqrt{\mu_\varphi\vphantom{L_\varphi}}}(z_{k+1}^i - z_k^i)
	\end{align}  
	\end{subequations}
	where $W_{ij} = [\bar W]_{ij} \otimes I_n$ and $x^{*}_j(w_k)$ is the components of solution of the dual subproblem shared by agent $j$.

	\begin{theorem}[Case $1$]\label{thm:case1} 
		Assume $F(x)$ is dual-friendly, \mbox{$L$-smooth} and $\mu$-strongly convex. Moreover, set $\tilde{\varepsilon} = \varepsilon/ R$. Then, after
		\begin{align*}
		N = O\left( \sqrt{\frac{L}{\mu}\chi(W)} \log \left( \frac{\sigma_{\max}(W)R^2}{\mu \epsilon} \right) \right) 
		\end{align*}
		iterations (oracle calls) of Eqs.\cref{per_agent}, the point $x^*(z_N)$ is an $(\varepsilon,\varepsilon/R )$-optimal solution to the optimization problem in Eq.~\eqref{consensus_problem2}.
	\end{theorem}
	\begin{proof}
		The desired result follows from the estimate in Eq.\cref{bound_dual}, the definition of $\chi(A^TA)$ and $A = \sqrt{W}$ and $\tilde \varepsilon =\varepsilon/R$.
	\end{proof}
	
	\subsection{$F(x)$ is $\mu$-strongly convex and $M$-Lipschitz}
	
	The assumption of $F(x)$ being \mbox{$\mu$-strongly} convex implies that the dual function $\varphi(y)$ is $L_\varphi$-smooth. In this case, we can use a regularization technique to induce the strong convexity of the dual function.
	
	\begin{definition}
		For a convex function $f(x)$, we define the strongly convex regularized function $f^\mu(x)$ with regularizer $r^2(x)$ as
		\begin{align}\label{regularizer}
		f^\mu(x) \triangleq f(x) + \frac{\mu}{2}r^2(x).
		\end{align}
	\end{definition}
	
	\begin{theorem}[Case $2$]\label{thm:case2}
		Assume $F(x)$ is dual-friendly, \mbox{$M$-Lipschitz} and $\mu$-strongly convex. Then,
		after
		\begin{align*}
		N = \tilde O\left( \sqrt{\frac{ M^2}{\mu\varepsilon} \chi(W)} \right)
		\end{align*}
		iterations (oracle calls) of Eqs.\cref{per_agent} on the regularized dual problem $\varphi^{\mu_\varphi}(y)$, with $\mu_\varphi = \varepsilon / R^2$ and regularizer $\|y\|_2^2$, the point $x^*(z_N)$ is an $(\varepsilon,\varepsilon/R )$-optimal solution to the optimization problem in Eq. \eqref{consensus_problem2}.
	\end{theorem}
	\begin{proof}
		Initially, lets construct a regularized dual problem. Then, we have that
		\begin{align}\label{dual_regularized}
		\min_y \varphi^{\hat\mu}(y) & = \varphi(y) + \frac{\hat\mu}{2}\|y\|_2^2
		\end{align}
		where $\hat\mu = \frac{\varepsilon}{R^2}$. Assume there exists $y_N$ such that
		\begin{align*}
		\varphi^{\hat\mu}(y_N) - \varphi^{\hat\mu}_* &\leq \frac{\varepsilon}{2}
		\end{align*}
		where $\varphi^{\hat\mu}_*$ is the optimal value of the regularized dual function. Then
		\begin{align*}
		\varphi(y_N) - \varphi^* &\leq \varepsilon.
		\end{align*}
		where $\varphi^*$ is the optimal value of the dual function.
		
		It follows from \cite{gas16b} that the number of oracle calls required to find a $(\varepsilon,\tilde{\varepsilon})$-optimal solution is
		\begin{align*}
		O\left(  \sqrt{\frac{2L_\varphi(\varepsilon + 2R \tilde{\varepsilon})}{\tilde{\varepsilon}^2}}\log\left(\frac{4L_\varphi \left(  \min\limits_{Ax=0}F(x)-\min\limits_x F(x))(\varepsilon+2R\tilde{\varepsilon}\right) }{\varepsilon\cdot \tilde{\varepsilon}^2}\right)\right)
		\end{align*}
		
		Additionally, from Theorem $3$ in \cite{lan17} it follows that
		\begin{align}\label{Soomin_bound}
		R^2 & = \|y^*\|_2^2 \leq \frac{\|\nabla F(x^*)\|_2^2}{\sigma_{\min}(A)} = \frac{M^2}{\sigma_{\min}(A)}.
		\end{align}
		
		The final results follow by the definitions of $\tilde\varepsilon$ and $L_\varphi$. 
		
	\end{proof}
	Note that we typically do not know $R^2 = \|y^*\|_2^2$. Thus, we require a method to estimate the strong convexity parameter $\hat{\mu}$ which is challenging \cite{nes07,odo15}. Therefore, we can apply the restarting technique on $\mu$ \cite{odo15}. The payment for that is just an $8$ multiplicative factor in the estimation \cite{iou14}. Similarly, a generalization of the FGM algorithm can be proposed when $L_\varphi$ is unknown \cite{ani17}. The specific details of this generalizations are out of the scope of this paper.
	
	\subsection{$F(x)$ is \mbox{$L$-smooth}}
	
	In this case, we can follows the same regularization technique as in the previous scenario. However, in this case we can regularize the primal function.
	
	\begin{theorem}[Case $3$]\label{thm:case3}
		Assume $F(x)$ is dual-friendly and $L$-smooth convex. Then, after
		\begin{align*}
		N =\tilde O\left( \sqrt{\frac{ LR_x^2}{\varepsilon} \chi(W)} \right)
		\end{align*}
		iterations (oracle calls) of Eqs.\cref{per_agent} on the dual problem of the regularized function $F^\mu(x)$ with $\mu = \varepsilon / R_x^2 = \varepsilon / \|x^* - x^*(z_0)\|_2^2$ and regularizer $\|x - x^*(z_0)\|_2^2$, the point $x^*(z_N)$ is an \mbox{$(\varepsilon,\varepsilon/R )$-optimal} solution to the optimization problem in Eq. \eqref{consensus_problem2}.
	\end{theorem}
	
	\begin{proof}
		We can regularize the primal problem in Eq. \eqref{main2} such that we instead try to solve
		\begin{align*}
		\min_{Ax=0} F(x) + \frac{\mu}{2}\|x -x^*(z_0)\|_2^2,
		\end{align*}
		with $\mu = \frac{\varepsilon}{R^2_x}$. Additionally, if the minimizer $x^*$ is not unique we will consider it as the closest solution from $x^*(z_0)$, i.e. closest to the initial point of the algorithm.
		
		This take us back to the case of  \cref{sec:strong_smooth} for which we can conclude that up to logarithmic factor that the number of iterations required is:
		\begin{align}\label{result_smooth}
		\tilde{O}\left( \sqrt{\frac{L R_x^2}{\varepsilon}\chi(A^TA)}\right).
		\end{align}
	\end{proof}
	
	\subsection{$F(x)$ is $M$-Lipschitz}
	
	\begin{theorem}[Case $4$]\label{thm:case4}
		Assume $F(x)$ is dual-friendly and $M$-Lipschitz. Then, after
		\begin{align*}
		N = \tilde O\left( \sqrt{\frac{ M^2R_x^2}{\varepsilon^2} \chi(W)} \right) 
		\end{align*}
		iterations (oracle calls) of Eqs.\cref{per_agent} on the regularized dual problem $\varphi^{\mu_\varphi}(y)$, with $\mu_\varphi = \varepsilon / R^2$ and regularizer $\|y\|_2^2$, of the regularized function $F^\mu(x)$ with $\mu  = \varepsilon / R_x^2$ and regularizer $\|x - x^*(z_0)\|_2^2$, the point $x^*(z_N)$ is an \mbox{$(\varepsilon,\varepsilon/R )$-optimal} solution to the optimization problem in Eq. \eqref{consensus_problem2}.
	\end{theorem}
	\begin{proof}
		Similarly as in the cases above we will use regularization. We will regularize the primal problem with $\mu = \frac{\varepsilon}{R^2_x}$ and then regularize the dual problem too with $\hat{\mu} = \frac{\varepsilon}{R^2}$. Therefore, up to logarithmic factors the \cref{alg:nesterov2} will stop in no more than 
		\begin{align}\label{result2_convex}
		O\left( \sqrt{\frac{\sigma_{\max}(A)R^2_x(\varepsilon+2R\tilde{\varepsilon})}{\varepsilon\tilde\varepsilon^2}}\right)
		\end{align}
		iterations. Where the estimate in Eq. \eqref{result2_convex} is unimporvable up to logarithmic factors.
	\end{proof}
	
	\subsection{A summary}
	
	\cref{tab:summary} presents an informal summary of the results presented in this section. In particular, it shows the number of oracle calls required in each of the problems to obtain an $\varepsilon$-optimal solution. Moreover, it shows the specific dependency of the convergence rates regarding the properties of the functions considered. 
	\begin{table}[H]
		\centering
		\begin{tabular}{|c|c|c|} \hline
			\bf Property of $F(x)$ & \bf Iterations Required  \\ \hline
			$\mu$-strongly convex and $L$-smooth & $\tilde{O}\left( \sqrt{\frac{L}{\mu} \chi(W)} \right)$          \\ 
			$\mu$-strongly convex and $M$-Lipschitz&$\tilde{O}\left( \sqrt{\frac{M^2}{\mu \varepsilon} \chi(W)}\right) $ \\
			$L$-smooth &   $\tilde{O}\left( \sqrt{\frac{L R^2_x}{\varepsilon} \chi(W)}\right) $ \\
			$M$-Lipschitz& $\tilde{O}\left( \sqrt{\frac{M^2R_x^2}{\varepsilon^2} \chi(W)}\right) $\\ \hline
		\end{tabular}
		\caption{A summary of Algorithmic Performance}
		\label{tab:summary}    
	\end{table}
	
	The specific value of $\chi(W)$ and its dependency on the number of nodes $m$ has been extensively studied in the literature of distributed optimization \cite{ned15}. In \cite{ned17b}, Proposition $5$ provides an extensive list of worst-case dependencies of the spectral gap for large classes of graphs. Particularly, for fixed undirected graphs, one can construct a set of weights $W$ for which $\chi(W) = O(m^2)$ \cite{ols14}. This matches the best upper bound found in the literature of consensus and distributed optimization \cite{ore10,liu13,ols14}. As an immediate conclusion it is clear that the form in Eq.\eqref{consensus_problem2} with the constraint described as \mbox{$\sqrt{W}x=0$} should be preferred over the description as $Wx=0$, even though both representation correctly describe the consensus subspace $x_1 = \hdots x_m$. Particularly, when we pick $A = \sqrt{W}$, we have \mbox{$\chi(A^TA) = \chi(W)$} instead of \mbox{$\chi(W^TW) = \chi(W^2) \gg \chi(W)$}. For example for a star graph \mbox{$\sqrt{\chi(W)} =  O( \sqrt{m})$}, for complete graphs $\sqrt{\chi(W)} = O(1)$, for path graphs \mbox{$\sqrt{\chi(W)} \sim \text{diam}(\mathcal{G}) = O( m)$}, for regular networks $\sqrt{\chi(W)} \geq \frac{\text{diam}(G)}{\log m}$. We can observe that typically $\sqrt{\chi(W)} $ corresponds to the diameter of the graph $G$ and the square root of the spectral gap of the interaction matrix. 
	
	\section{The cost of communications}\label{sec:consensus}
	
	In this section, we will study a particular form of the optimization problem in Eq. \cref{main_problem} known as the consensus problem. This will help us further understand the execution and performance of the studied distributed optimization algorithms. Moreover, it will provide an answer to what is the additional cost, regarding algorithmic complexity, when we try to solve an optimization problem in a distributed manner. We refer to this cost as the cost of communication. We will show how, in the most basic setting, the cost of communications and it depends on the topology of the communication network over which the agents exchange information.
	
	Assume that each node in the graph $\mathcal{G}$ holds an initial numeric value {$x^i_0$} and can compute the {weighted} average of the values {held} by its neighbors at each iteration. {Let $x^i_k$ be the value of node $i$ at iteration $k$.} {We would like to know} how many iterations are required (and what is the proper algorithm) to reach consensus. Particularly, we will reach $\varepsilon$-consensus if for any $\varepsilon > 0$
	\begin{align*}
	\sqrt{\sum\limits_{i=1}^{{m}} \left(x^i_N - \frac{1}{{m}} \sum\limits_{j=1}^{{m}}{x^j_0}\right)^2} & \leq \varepsilon \sqrt{\sum\limits_{i=1}^{{m}} \left({x^i_0} - \frac{1}{{m}} \sum\limits_{j=1}^{{m}}{x^j_0}\right)^2}.
	\end{align*}
	
	This problem can be solved by considering the convex optimization problem
	\begin{align}\label{consensus}
	\min_{{x}} \frac{1}{2} \left\langle {x,Wx} \right\rangle.
	\end{align}
	where $W$ is a communication matrix as defined in Section \ref{sec:problem}.
	
	
	Theorem \ref{thm:case1} provides a direct estimate on the number of iterations required to reach consensus, particularly we need $O(\sqrt{\chi(W)}\log \varepsilon^{-1})$, where we have used the fact that \cref{consensus} is $\sigma_{\min}(\sqrt{W})$-strongly convex in  $x_0 +\ker(W)$ and has $\sigma_{\max}(\sqrt{W})$-Lipschitz continuous gradients. Moreover, it follows that this estimate cannot be improved up to constant factors. It is important to remark that the Fast Gradient Method, with $x^i_0 = \tilde y^i_0$, {does respect the information constraints imposed by the network}. That is, each node only requires the computation of the average of its immediate neighbors and aggregation can be performed in a fully distributed manner.
	
	The estimates presented in the \cref{thm:case1,thm:case2,thm:case3,thm:case4} cannot be improved up to logarithmic factors. Particularly in the smooth cases where $L < \infty$ these estimations follow (up to logarithmic factors) form the classical centralized complexity estimation of the FGM algorithm, where for a $\mu$-strongly convex and $L$-smooth convex problem we require at most
	\begin{align*}
	O\left(  \sqrt{\frac{L}{\mu} }\log \left( \frac{\mu R_x^2}{\varepsilon}\right) \right) 
	\end{align*}
	oracle calls and for a $L$-smooth convex problem we require at most
	\begin{align*}
	O\left( \sqrt{\frac{LR_x^2}{\varepsilon}}\right) .
	\end{align*}
	oracle calls.
	
	In addition to these centralized estimates we need to take into account the fact that one has to perform $\sqrt{\chi(W)} \log (\varepsilon^{-1})$ additional consensus steps at each iteration of the classical FGM. 
	
	The Laplacian matrix $\bar W$ can be written as $\bar W  = D - \tilde{A}$	where $\tilde{A}$ is the adjacency matrix of the graph $\mathcal{G}$ and $D$ is the degree matrix. In addition, $\tilde{A} = D P$ for a stochastic matrix $P$ (the transition matrix of a Markov Chain). Thus, one can apply a simple power method such as $x^{k+1} = Px^k$ to find the consensus value (see ergodic theorem for Markov Chains). Each node depends only on the computation of averages of its neighbor values because of the matrix-vector multiplication $Px^k$. As a result the number of required iterations is $O(\chi(W)\log \varepsilon^{-1})$ (\cite{ned09,tsi84,ber89}). This result can be considered as a non-accelerated weighted gradient method for the problem in Eq.~\eqref{consensus} and holds even if the graph is time-varying and directed (i.e. $\tilde{A}$ is not symmetric \cite{ned13,ned15b,ned16}). However, it is still an open question whether one can get accelerated rates as those provided by \cref{alg:nesterov2} when the graph is directed. 
	
	In the following sections, we explore extensions and open problems related to the results presented so far. We study the case when the functions are not \textit{dual-friendly}, when the function is \textit{proximal-friendly} and how to improve the condition number $L /\mu$.
		
	\section{No Explicit Dual Solution is available}\label{sec:no_dual}
	
	In all the results presented so far, we have assumed that the problem is \textit{dual-friendly} in the sense that we have readily available solutions to the dual subproblem. In this section, we will explore the case when this is not possible.
	
	Note that the dual subproblem can also be computed in a distributed manner. Specifically, we have that
	\begin{align*}
	\argmax_x \left\lbrace \left\langle A^Ty,x \right\rangle -f(x) \right\rbrace  & = \argmax_{x_i, i=1,\hdots,m} \left\lbrace \left\langle \sum_{j=1}^{m} A^T_{ij }y_j,x_i \right\rangle -f_i(x_i) \right\rbrace.
	\end{align*}
	
	When the function is smooth or when is strongly convex and smooth, one can solve the auxiliary problem
	\begin{align*}
	\max_x \left\lbrace \left\langle y ,Ax \right\rangle  -f(x) \right\rbrace  & = \left\langle y ,Ax^*(A^Ty) \right\rangle  -f(x^*(A^Ty))
	\end{align*}
	using fast gradient methods (if the function is smooth, one should make one additional regularization $\mu  = \frac{\varepsilon}{R^2_x}$) applied for the strongly convex problem. Therefore, we can find a solution for the dual problem, i.e. $x^*(A^Ty)$, in a logarithmic number of iterations from a desired relative precision $\delta$. This fact allows us avoid considering the error in the computation of $x^*(A^Ty)$. When the function is strongly convex and smooth, we can solve the auxiliary problem in $O\left(  \sqrt{\frac{L}{\mu}} \log(\delta^{-1})\right) $ oracle calls (oracle call is the calculation of $\nabla f(x)$). On the other hand, when the function is only smooth, we require $O\left(  \sqrt{ \frac{LR_x^2}{\varepsilon} }\log (\delta^{-1})\right)$ iterations. Both estimates are optimal up to logarithmic factor. Therefore, in those cases, we propose totally optimal methods.
	
	Unfortunately, this is not the case when the function $f(x)$ is not smooth. Nevertheless, in these cases we might use another approach, that gives optimal estimates in both senses: the total number of oracle calls (calculations $\nabla f(x)$) and the total number of communications ($Ax$, $A^Ty$ multiplications).
	
	Lets consider first the case where $f(x)$ is convex and apply Nesterov's smoothing technique \cite{nes05,nam14} to the dual problem
	\begin{align}\label{nes_smooth}
	\min_{Ax=0} f(x) & = \min_x \left\lbrace \max_y \left\lbrace \left\langle y,Ax \right\rangle -f(x) \right\rbrace  \right\rbrace  \nonumber \\
	& = \max_y \left\lbrace \min_x \left\lbrace f(x) + \left\langle y,Ax \right\rangle  \right\rbrace  \right\rbrace
	\end{align}
	
	Moreover, by using the bound in Eq. \eqref{Soomin_bound}, we can replace the function
	\begin{align*}
	G(Ax) = \max_y \left\langle y,Ax\right\rangle
	\end{align*}
	by
	\begin{align*}
	G_\varepsilon(Ax) = \max_y \left\lbrace \left\langle y,Ax \right\rangle  - \frac{\varepsilon}{2R^2} \|y\|_2^2\right\rbrace = \frac{R^2}{2\varepsilon} \|Ax\|_2^2.
	\end{align*}
	
	Finally, one can show that the function $G_\varepsilon(Ax)$ has a $\frac{\sigma_{\max}(A^T)R^2}{\varepsilon}$-Lipschitz continuous gradient, (note that $\sigma_{\max}(A^T) = \sigma_{\max}(A)$). So, we can solve the composite type mixed smooth/non-smooth type problem
	\begin{align}\label{composite}
	\min_{\|x\|_2 \leq R_x } \underbrace{G_\varepsilon(Ax)}_{\sim \frac{1}{\varepsilon}\text{-Lipschitz gradient}} + \underbrace{f(x)}_{M-\text{Lipschitz}}
	\end{align}
	where gradient oracle for $G_\varepsilon(Ax)$ requires a constant number of $Ax$ multiplications (because we can explicitly write the formula for $\nabla G_\varepsilon (z)$) and the gradient oracle for $f(x)$ requires one $\nabla f(x)$ calculations. Using Lan's accelerated gradient sliding \cite{lan16}, one can find an $\varepsilon$-solution (in function value) of Eq. \eqref{composite} without any auxiliary dual problem, after
	\begin{align*}
	N_{Ax} &= O\left(  \sqrt{\frac{\frac{\sigma_{\max}(A)R^2}{\varepsilon}R_x^2}{\varepsilon}} \right)
	= O\left( \sqrt{\frac{M^2R_x^2}{\varepsilon^2} \chi(A^TA)}\right) 
	\end{align*}
	oracle calls, i.e., $Ax$ multiplications and
	\begin{align*}
	N_{\nabla f(x)} &= O\left(  \frac{M^2 R_x^2}{\varepsilon^2}\right) 
	\end{align*}
	oracle calls, i.e., $\nabla f(x)$-gradient calculations. Unfortunately, in this approach we can guarantee $\|Ax_N\|\leq \frac{\varepsilon}{R}$ only in the best case \cite{ani15}.
	
	Using the restart technique \cite{iou14,jud11} one can extend Lan's accelerated gradient sliding for $f(x)$ being $\mu$-strongly convex. At the $k$-th restart the number of oracle calls is \mbox{$N_{Ax}=O\left(  \sqrt{\frac{\sigma_{\max}(A)R^2}{\mu\varepsilon}}\right)$} computations of $Ax$ and $N_{\nabla f(x)} = O\left(  \frac{2^kM^2 }{\varepsilon^2 R_x^2}\right) $ oracle calls for $\nabla f(x)$. This allows to improve estimates for the problem in Eq. \eqref{composite} in the following manner:
	\begin{align}\label{prox1}
	N_{Ax} & = O\left( \sqrt{\frac{M^2}{\mu \varepsilon} \chi(A^TA)} \log \left( \frac{\mu R^2_x}{\varepsilon}\right)  \right) 
	\end{align}
	computations of $Ax$ and
	\begin{align}\label{prox2}
	N_{\nabla f(x)} & = O \left( \frac{ M^2 }{\mu\varepsilon}\right) 
	\end{align}
	oracle calls of $\nabla f(x)$, where these estimates are optimal up to logarithmic factors. Moreover, one can extend these results to stochastic optimization problems and the estimations will not change \cite{lan17}.
	
	\section{Acceleration when $f(x)$ is a \textit{proximal-friendly} functional}\label{sec:prox}
	
	Let us consider the case when $f(x)$ is convex, but not strongly convex nor smooth, and return to Eq. \eqref{nes_smooth} which can be rewritten as
	\begin{align*}
	\min_{Ax=0} f(x) & = \max_y \min_x \left\lbrace  f(x) + \left\langle y,Ax \right\rangle \right\rbrace  \\
	& = \max_y \min_z \underbrace{\min_x \left\lbrace f(x) + \left\langle y,Ax \right\rangle + \frac{1}{2} \|x-z\|_2^2 \right\rbrace }_{G(y,z)}.
	\end{align*}
	
	We will say the function $f(x)$ is \textit{proximal-friendly}, if  we can solve the auxiliary strongly convex problem
	\begin{align*}
	\text{prox}_{f,A^Ty} (z) & = \argmin_x \left\lbrace f(x) + \left\langle y,Ax \right\rangle + \frac{1}{2}\|x-z\|_2^2\right\rbrace 
	\end{align*}
	explicitly, similarly as the definition of \textit{dual-friendly}.
	
	Therefore, given that
	\begin{align*}
	\|\nabla G(y',z') - \nabla G(y,z) \|_2 & \leq L_y\|y' - y\|_2 +L_z\|z' - z\|_2 ,
	\end{align*}
	one can find an $\varepsilon$- solution (in terms of the duality gap) of the saddle point problem
	\begin{align*}
	\max_{\|y\|_2\leq R} \min_z G(y,z)
	\end{align*}
	after $N_{Ax} = O\left( \frac{1}{\varepsilon}\right) $, $Ax$ multiplications, see \cite{bub15}, section $5.2$.
	
	Unfortunately, this does not provide any acceleration. Since, to find $\text{prox}_{f,A^Ty}(z)$ typically, one has to find an $\varepsilon$-solution of a strongly convex optimization problem. This can be done in $O(\frac{1}{\varepsilon})$ calculations of $\nabla f(x)$. So the total number of iterations will be of the order $N_{\nabla f(x)} = O\left(  \frac{1}{\varepsilon^2}\right) $.
	
	Now, let us propose another approach, consider the original problem in Eq. \eqref{main2} and add the additional term $\frac{1}{2}\|Ax\|_2^2$ which is identically zero on any feasible point $Ax=0$. Thus,
	\begin{align*}
	\min_{Ax=0} f(x) & = \min_{Ax=0} \left\lbrace f(x) + \frac{1}{2}\|Ax\|_2^2\right\rbrace  \\
	& = \min_{x,z;Ax=Az} \left\lbrace  f(x) + \max_y \left\langle -y,Az \right\rangle + \frac{1}{2}\|Ax\|_2^2 \right\rbrace  \\
	& =\min_{x,z} \left\lbrace f(x) + \max_y \left\langle -y,Az \right\rangle +  \max_{y'} \left\langle y',Az-Ax \right\rangle + \frac{1}{2}\|Ax\|_2^2 \right\rbrace  \\
	& = - \min_{y,y'} \left\lbrace  \max_x  \left( \left\langle A^Ty',x \right\rangle -f(x)\right)  +\max_{u \in \text{Im}A} \left\lbrace  \left\langle y-y',u \right\rangle    - \frac{1}{2} \|u\|^2_2 \right\rbrace \right\rbrace  ^{Az=u}\\
	& = - \min_{y,y'} \left\lbrace \varphi(y') + \frac{1}{2}\| \text{proj}_{\ker(A^T)^{\perp} }(y - y')  \|_2^2\right\rbrace 
	\end{align*}
	
	If  $\varphi(y') $ is proximal friendly, we can solve the auxiliary problem (this is the proximal version of the standard dual problem) explicitly, then using accelerated proximal gradient descent (in the space of $y$) one can find an $\varepsilon$ solution to the dual problem in $  O\left( \varepsilon^{-\frac{1}{2}}\right) $ proximal steps \cite{par14}. Moreover, the results in \cite{lin15} allows for an extension to randomized methods. Therefore, we obtain acceleration. 
	
	
	\section{Improving the condition number $\frac{L}{\mu}$ when $F(x)$ is strongly convex and smooth}\label{sec:condition}
	
	Considering the problem in Eq. \eqref{main_problem}, if each $f_i(x_i)$ is a $\mu_i$-strongly convex function , and has $L_i$-Lipschitz continuous gradients, then $F(x) $ is a \mbox{$\mu = \min_{i=1,\hdots,m}\mu_i$-strongly} convex and has $L = \max_{i=1,\hdots,m}L_i$-Lipschitz continuous gradients.
	
	As a result, the condition number $\frac{L}{\mu}$ can be large in general, at least if one of the $\mu_i$ is small. To overcome this drawback, we can formulate another regularization for the problem in Eq. \eqref{consensus_problem2} as
	\begin{align}\label{better_cond}
	\min_{\sqrt{W} x=0} F_\alpha (x) & = \sum\limits_{i=1}^{m}f_i(x_i) + \frac{\alpha}{2}\left\langle x,Wx\right\rangle
	\end{align}
	
	Furthermore, the regularized function $F_\alpha$ is $\mu \geq \min \left\lbrace \sum\limits_{i=1}^{m}\mu_i,\alpha \lambda_{\min}(W)\right\rbrace $-strongly convex and has \mbox{$L \leq  \left( \max\limits_{i=1,\hdots,m}L_i + \alpha \lambda_{\max}(W)\right) $-Lipschitz} continuous gradient. Moreover, if we set \mbox{$\alpha \simeq
		\sum\limits_{k=1}^{m}\mu_k / \lambda_{\min}(W)$}, one can solve the problem in Eq. \eqref{better_cond} with relative precision $\varepsilon$ after (see also \cite{ned16w})
	\begin{align*}
	N_{Wx} & =O\left(  \left( \frac{\max\limits_{i=1,\hdots,m}L_i}{\sum\limits_{i=1}^{m}\mu_i} + \chi(W)\right) \sqrt{\chi(W)} \log^2(\varepsilon^{-1})\right) 
	\end{align*}
	communication steps and
	\begin{align*}
	N_{\nabla F(x)} = O\left( N_{Wx}   \log^{-1}(\varepsilon^{-1})\right) 
	\end{align*}
	gradient $\nabla F(x)$ calculations.
	
	This estimate shows that we can replace the smallest strong convexity constant for the sum among all of them, but we have to pay an additive price proportional to the condition number of the graph.
	
	Using the regularization technique with $\mu_i = \frac{\varepsilon}{mR^2_{x_i}} = \frac{\varepsilon}{R^2_x}$ one can extend this result to the case where the function is just smooth.

	\section{Discussion and Conclusions}\label{sec:discussion}
	
	In this section, we will discuss a set of possible extensions and open problems related to the results presented so far. Particularly, we will explore the case when the graph is directed or changing with time. What to do if the time required for by the oracle is not comparable with the communication steps. How to handle non-Euclidean setups. 
	
	\noindent\textbf{Time-Varying/Directed Graphs:} Throughout this paper, we have assumed that the considered network of agents is fixed and undirected. This implies that the construction of the interaction matrix can be done in a distributed manner and the result will be a symmetric matrix, i.e., $W =W^T$. In a very specific case, when we have two communication networks, dual to each other, our approach will work, but this is a rare situation. Non-accelerated approaches have been shown successful in the study of time-varying directed graphs \cite{ned13,ned16,ned15b,ned16f}. They have been shown to provide solutions to the distributed optimization problem for strongly convex and smooth functions with linear rates. Nevertheless, it is an open problem to show that optimal convergence rates, comparable to those in the centralized case can be achieved for time-varying and directed graphs.
	
	Let us return to the consensus problem in Eq. \eqref{consensus}. But now assume that from time to time the matrix $W$ changes, nonetheless remaining a interaction matrix. Therefore, we have a family of nonnegative semi-definite quadratic functions with the same kernel, $\ker(W)$ (in our case described as $x_1 =  \hdots  = x_k$). How can we find a projection of $x_1^0,\hdots,x_m^0$ on this set working at each iteration with different matrices $W$? One can solve the problem in Eq. \eqref{consensus} with relative precision by the simple gradient descent method on $\chi(W)\log \varepsilon^{-1}$ communication steps because this dynamic has a Lyapunov function: that is the square distance between the current point and $\ker(W)$. Nevertheless, it is known that one can accelerate this value to $\sqrt{\chi(W)}\log \varepsilon^{-1}$ for fixed graphs. Whether acceleration is possible for time-varying graphs remains an open question to the best of the authors' knowledge. Such generalization of  Nesterov's accelerated gradient method is unknown \cite{nar05,bub15,all14,lin15,tay17,hu17,sun17}. On the other side, if one can detect the moment the graph changes, and such changes don't happen very often one can use restarting techniques \cite{fer16,gas17}, which provide the following estimations for the number of communication steps needed
	\begin{align*}
	\sqrt{\max_W \left( \frac{ \sigma_{\max}(\sqrt{W})}{\sigma_{\min}(\sqrt{W})}\right) } \log \varepsilon^{-1}.
	\end{align*}
	
	
	\noindent\textbf{Fast Gossip Steps:} CPUs in these days can read and write from and to memory at over $10$GB per second, whereas communication over $TCP/IP$ is about $10MB$ per second \cite{lan17}. Therefore, the gap between intra-node computation and inter-node communication is about $3$ order of magnitude. Communication start-up cost itself is also not negligible as it usually takes a few milliseconds. Let us consider that one node can calculate $\nabla f_i (x_i)$  (or even calculate $x^*_i(A^Ty)$) in $1$ unit of time and the communication step takes $\tau$ units of time. In case $\tau \gg 1$, in this regime all the results above seem reasonable because we first think of communication steps. However, if $\tau \ll 1$ one should use a method with multiple communication steps. Chebyshev acceleration is suggested in \cite{sca17}, where instead of using the interaction matrix $W$ one can analyze the algorithm with a different one $P_K(W)$. $P_K(W)$ is a polynomial of degree $K$, and its eigengap is maximized with a specific choice based on Chebyshev polynomials. If $K$ is chosen as $\sqrt{\chi(W)}$ then \mbox{$\chi\left( P_{\sqrt{\chi(W)}}(W)\right)  \sim 1$}.
	
\noindent\textbf{$p$-norms, with $p\geq 1$:} The cases when $F(x)$ is convex or strongly convex can be generalized to $p$-norms, with $p\geq 1$ see \cite{ani17}. Particularly, the definitions of the condition number $\chi(\cdot)$ needs to be defined accordingly. Let's introduce a norm \mbox{$\| x\|^2_p = \|x_1\|_p^2 + ... + \|x_m\|_p^2$} for \mbox{$p\geq 1$} and assume that $F(x)$ is $\mu$-strongly convex and $L$-Lipschitz continuous gradient in this (new) norm $\|\cdot \|_p$ (in $\mathbb{R}^{mn}$), see \cite{nes15}~(Lemma 1), \cite{dvu17b}~(Lemma 1) and \cite{nes05}~(Theorem 1).. Thus 
\begin{align*}
\chi(W) = \frac{\max_{\|h\|=1} \frac{<h,Wh>}{\mu}}{\min_{\|h\|=1, h \perp \text{ker}(W)} \frac{<h,Wh>}{L}}
\end{align*} 
	
	One can try to generalize this results to an intermediate level of smoothness. That is, try to propose the method for arbitrary H\"older parameter $\nu \in [0,1]$. For example, one can use Universal Nesterov's method by skipping the adaptation and proper choosing of $\delta(\nu,\epsilon)$. This is another way to obtain results in the non-smooth case as a special situation $\nu = 0$. In the dual space, we will not have classical strong convexity but just uniform convexity. However, it can be studied by introducing inexact oracle as in \cite{gas15}.
	
	\vspace{1em}
	
	We have provided convergence rate estimates for the solution of convex optimization problems in a distributed manner. The provided complexity bounds depend explicitly on the properties of the function to be optimized. If $F(x)$ is smooth, then our estimates are optimal up to logarithmic factors otherwise our estimates are optimal up to constant factors. The inclusion of the graph properties in the form of $\sqrt{\chi(W)}$ shows the additional price to be paid in contrast with classical (centralized/non-distributed) optimal estimates. The authors recognize that the proposed algorithms required, to some extent, some global knowledge about the graph properties and the condition number of the global function, nevertheless we aim to provide a theoretical foundation for the performance limits of the distributed algorithms. Further explorations of the cases where such information is not available require additional study.

	\bibliographystyle{plain} 
	\bibliography{IEEEfull,opt_dec2,bayes_cons_3}

\end{document}